\documentclass[reqno]{amsart}
\usepackage{amsthm}
\usepackage{amssymb}
\usepackage{algorithmic} 
\usepackage{algorithm} 
\usepackage{titlesec} 
\usepackage{graphics}
\usepackage{graphicx}
\usepackage{tikz}

\titleformat{\section}[hang]{ \bf}
  {\thesection}{0.5em}{\large}

\titleformat{\subsection}[runin]{\bf}
{\thesubsection}{0.5em}{}

\newtheorem{proposition}{Proposition}[section]
\newtheorem{theorem}[proposition]{Theorem}
\newtheorem{corollary}[proposition]{Corollary}
\newtheorem{lemma}[proposition]{Lemma}
\newtheorem{definition}[proposition]{Definition}

\newtheorem{example}[proposition]{Example}
\def\argmin{ \mathop{{\rm argmin}}}

\newcommand{\dist}{\mathrm{dist}}
\newcommand{\dom}{\mathrm{dom}\,}
\newcommand{\epi}{\mathrm{epi}\,}
\def\elim{{\rm e-}\hspace{-0.5pt}\lim}

\newcommand{\gph}{\mathrm{gph}\,}

\def\Lim{\mathop{{\rm Lim}\,}}
\def\Liminf{\mathop{{\rm Lim}\,{\rm inf}}}
\def\Limsup{\mathop{{\rm Lim}\,{\rm sup}}}
\newcommand{\nul}{\mathrm{nul}\,}
\newcommand{\p}{\partial}
\newcommand{\R}{\mathbb R}
\newcommand{\rank}{\mathrm{rank}\,}
\newcommand{\rbar}{\overline{\mathbb R}}

\newcommand{\too}[3]{#1\to_{#2}#3}
\newcommand{\bR}{\mathbb{R}}
\newcommand{\bB}{\mathbb{B}}
\newcommand{\bN}{\mathbb{N}}
\newcommand{\eR}{\overline{\R}}
\newcommand{\map}[3]{#1 :#2\rightarrow #3}
\newcommand{\ip}[2]{\left\langle #1,\, #2\right\rangle}
\newcommand{\half}{\frac{1}{2}}
\newcommand{\norm}[1]{\left\Vert #1\right\Vert}
\newcommand{\tnorm}[1]{\left\Vert #1\right\Vert}
\newcommand{\snorm}[1]{\left\Vert #1\right\Vert_*}
\newcommand{\onorm}[1]{\left\Vert #1\right\Vert_1}
\newcommand{\set}[2]{\left\{#1\,\left\vert\; #2\right.\right\}}
\newcommand{\ncone}[2]{N\left(#1\,\left|\, #2\right.\right)}

\newcommand{\eplqs}{{\scriptscriptstyle{(U,B,R,b)}}}

\newcommand{\bu}{{\bar{u}}}
\newcommand{\bv}{{\bar{v}}}
\newcommand{\bx}{{\bar{x}}}
\newcommand{\by}{{\bar{y}}}

\newcommand{\ty}{{\tilde{y}}}
\newcommand{\tx}{{\tilde{x}}}

\begin{document}


\title[Epi-convergent Smoothing]{Epi-convergent Smoothing with Applications to Convex Composite Functions}

.

\author{James V. Burke}
\address{University of Washington, Department of Mathematics, Box 354350
Seattle, Washington 98195-4350}
\email{burke@math.washington.edu}

\author{Tim Hoheisel}
\address{University of W\"urzburg, Institute of Mathematics,  
 Campus Hubland Nord, Emil-Fischer-Stra\ss e 30, 97074 W\"urzburg, Germany}
\email{hoheisel@mathematik.uni-wuerzburg.de}
\thanks{This research was partially supported by the DFG (Deutsche 
Forschungsgemeinschaft) under grant HO 4739/1-1.}

\begin{abstract}
Smoothing methods have become part of the standard tool set for the study and solution of
nondifferentiable and constrained optimization problems as well as a range of other
variational and equilibrium problems. In this note we synthesize and extend recent results due to
Beck and Teboulle on infimal convolution smoothing for convex functions
with those of X.~Chen on gradient consistency for nonconvex functions.
We use epi-convergence techniques to define a notion of epi-smoothing that allows us to 
tap into the rich variational structure of the subdifferential calculus for nonsmooth, nonconvex, and
nonfinite-valued functions. As an illustration of the versatility and range of epi-smoothing techniques,
the results are applied to the general constrained optimization for which nonlinear programming
is a special case.
\end{abstract}

\keywords{smoothing method, subdifferential calculus, epi-convergence, infimal convolution, 
Moreau envelope, convex composite function, Karush-Kuhn-Tucker conditions}

\subjclass[2010]{49J52, 49J53, 90C26, 90C30, 90C46}

\maketitle

\pagestyle{myheadings}
\thispagestyle{plain}

\section{Introduction}\label{sec:intro}

A standard approach to solving nonsmooth and constrained optimization problems
is to solve a related sequence of unconstrained smooth approximations
\cite{BeT 12,BeT 89,Ber 73,Che 12,ENW 95,FiM 87,Hal 76,Nes 05,RoW 98}.
The approximations are constructed so that cluster points of the
solutions or stationary points of the approximating smooth problems  
are solutions or stationary  points for
the limiting nonsmooth or constrained optimization problem. In the setting of convex programming,
there is now great interest in these methods in the very large-scale
setting (e.g., see \cite{Donoho1992,LMN 07,Nes 05,Nes 07}), 
where first-order methods 
for convex nonsmooth optimization have been very successful. At the same time, there
are many recent applications of smoothing methods to general nonlinear
programming, equilibrium, and mathematical programs with equilibrium constraints, e.g., see
\cite{BiC 12,CC 99,ChM 95,Che 00,ChF 04,CNQ 00,FJQ 99,FLT 02,FP 99,GM 97}.
This paper is concerned with synthesizing and expanding the ideas presented in
two important recent papers on smoothing. The first is by Beck and Teboulle \cite{BeT 12} which
develops a smoothing framework for nonsmooth convex functions based on
{\em infimal convolution}. The second is by Chen \cite{Che 12}
which, among other things, studies the notion of {\em gradient consistency} for 
smoothing sequences. Our goal is to extend the ideas presented in \cite{BeT 12}
for convex functions to the class of {\em convex composite} functions and provide
conditions under which this extension preserves the gradient consistency.
Our primary tool in this analysis is the
notion of variational convergence called {\em epi-convergence} \cite{Att 84,AtW 89,RoW 98}.
Epi-convergence is ideally suited to the study of the variational properties of parametrized families
of functions allowing, for example, the development of a calculus of smoothing functions
which is essential for the applications to the nonlinear inverse problems that we have in mind
\cite{Aravkin2011,AravkinIFAC,ABP 11}. 
Epi-smoothing is a weaker notion of smoothing than those considered in 
\cite[Definition 2.1]{BeT 12} where
complexity results are one of the key contributions \cite[Theorem 3.1]{BeT 12}. 
It is the complexity results that require stronger
notions of smoothing. On the other hand, our goal is to establish limiting variational properties
in nonconvex applications, in particular, gradient consistency 
(see \cite[Theorem 1]{Che 12} and \cite[Theorem 4.5]{BHK 12}).

\noindent
We begin in Section \ref{sec:prelim} by introducing the notions of epigraphical 
and set-valued convergence upon which our analysis rests. We also introduce the tools
from subdifferential calculus \cite{RoW 98} that we use to establish gradient consistency.
In Section \ref{sec:smooth}, we define {\em epi-smoothing functions}  and develop a calculus
for these smoothing functions that includes basic arithmetic operations as well as composition.
In Section \ref{sec:InfConv}, we give conditions under which 
the Beck and Teboulle \cite{BeT 12} approach to
smoothing via infimal convolution also gives rise to epi-smoothing functions that satisfy
gradient consistency.
These results are then applied to {\em Moreau envelopes} (e.g., see \cite{RoW 98})
and {\em extended piecewise linear-quadratic functions}.
In Section \ref{sec:ConvCom}, we introduce convex composite functions an give conditions
under which the epi-smoothing results of Section \ref{sec:InfConv} can be extended
to this class of functions. In Section \ref{sec:ConOpt}, we conclude by applying the smoothing results for
convex composite functions to general nonlinear programming problems.

{\em Notation:} Most of the notation used is standard. An element 
$x\in\mathbb R^n$   is understood as a column vector, and
$\overline{\mathbb R}:=[-\infty,+\infty]$ is the {\em extended real-line}. The space of all real 
$m\times n$-matrices is denoted by $\mathbb R^{m\times n}$, and for 
$A\in \mathbb R^{m\times n}$, $A^T$ is its transpose. 
The {\em null space} of  $A$  is the set 
$$
\nul A:=\{x\in\mathbb R^n\mid Ax=0\}.
$$
By $I_{n\times n}$ we mean the ${n\times n}$ identity matrix and by $\mathsf{ones}(n,m)$
the $n\times m$ matrix each of whose entries is the number $1$.

\noindent
Unless otherwise stated, $\|\cdot\|$
denotes the {\em Euclidean norm} on $\mathbb R^n$
and $\onorm{\cdot}$ denotes the {\em $1$-norm}. 
If $C\subset\bR^n$ is nonempty and closed, the {\em Euclidean distance function} for $C$
is given by
\begin{equation}\label{eq:dist}
\dist(y\mid C):=\inf_{z\in C}\norm{y-z}.
\end{equation}
When $C$ is convex it is easily established that the distance function is a convex function, and
the optimization \eqref{eq:dist} has a unique solution $\Pi_C(y)$
which is called the {\em projection of $y$ onto $C$}.

\noindent
For a sequence $\{x^k\}\subset \mathbb \R^n$ and a (nonempty) set  $X\subset\mathbb R^n$ 
we abbreviate the fact that 
$x^k$ converges to $\bar x\in\mathbb R^n$ and $x^k\in X$ for all $k\in\mathbb N$ by 
$$
\too{x^k}{X}{\bar x}. 
$$
Moreover, for a function $f:\mathbb R^n\to \R$, define 
$$
\too{x^k}{f}{\bar x}
   \quad :\Longleftrightarrow \quad  x^k \to \bar x
   \quad{\rm and }\quad f(x^k)\to f(\bar x).
$$
This type of convergence 
coincides with ordinary convergence when $f$ is continuous.
\\
For a  real-valued function $f:\mathbb R^n\to \mathbb R$ differentiable at 
$\bar x$, the {\em gradient} is given by $\nabla f(\bar x)$ which is understood as  a
column vector. For a function $F:\mathbb R^n\to\mathbb R^m$ differentiable 
at $\bar x$, the {\em Jacobian}  of $F$ at $\bar x$ is denoted by $F'(\bar x)$, i.e.,
$$
   F'(\bar x)=
   \left(\begin{array}{c}
   \nabla F_1(\bar x)^T \\ 
   \vdots\\
   \nabla F_m(\bar x)^T 
   \end{array}\right) \in \mathbb R^{m \times n}.
$$

\noindent
In order to distinguish between single- and set-valued maps, we write $S:\mathbb R^n\rightrightarrows\mathbb R^m$ to indicate that 
$S$ maps vectors from $\mathbb R^n$ to subsets of $\mathbb R^m$.  The {\em graph} of $S$ is the set
$$
\gph S:=\{(x,y)\mid y\in S(x)\},
$$
which is equivalent to the classical notion when $S$ is single-valued.

\section{Preliminaries}\label{sec:prelim}

\noindent
In this section we review certain concepts from variational and nonsmooth analysis 
employed in the subsequent analysis. The notation is primarily 
based on \cite{RoW 98}.

\noindent
For an extended real-valued function  $f:\R^n\to \R\cup\{+\infty \}$ its  {\em epigraph}  is given by
$$
\epi f:=\{(x,\alpha)\in\mathbb R^n\times \mathbb R\mid f(x)\leq \alpha\}, 
$$
and its {\em domain} is the set
$$
\dom f:=\{x\in \mathbb R^n\mid f(x)<+\infty\}.
$$
The notion of the epigraph allows for very handy definitions of a number of properties 
for extended real-valued functions (see \cite{HUL 96,Roc 70,RoW 98}).

\begin{definition}[Closed, proper, convex functions]\label{def:LscCon}
A function  $f:\R^n\to \R\cup\{+\infty \}$ is called {\em lower semicontinuous (lsc)} (or {\em closed}) 
if $\epi f$ is a closed set.   
$f$ is called {\em convex} if $\epi f$ is a convex set. A convex function $f$ is said to be
{\em proper} if there exists $x\in\dom f$ such that $f(x)\in\bR$.
\end{definition}

\noindent
Note that these definitions coincide with the usual concepts for ordinary real-valued functions. 
Moreover, it holds that 
a convex function is always (locally Lipschitz) continuous on the 
(relative) interior of its domain \cite[Theorem 10.4]{Roc 70}.

\noindent
Furthermore, we  point out that, in what follows, for  an lsc, convex function 
$f:\mathbb R^n\to\mathbb R\cup\{+\infty\}$, we always exclude the case
$f\equiv +\infty$, which means that we deal with proper functions. 

\noindent
An important function in this context is the  {\em (convex) indicator function} of a  set
$C\subset\mathbb R^n$ given by  $\delta(\cdot\mid C):\mathbb R^n\to\mathbb R\cup\{+\infty\}$ with 
$$
\delta(x\mid C)=\left\{\begin{array}{rcl} 
0 & {\rm if} & x\in C,\\
+\infty & {\rm if} &  x\notin C.
 \end{array}\right.
$$
\noindent
The indicator function $\delta(\cdot\mid C)$ is convex if and only if $C$ is convex, 
and $\delta(\cdot\mid C)$ is lsc if and only if $C$ is closed.
\\
\\
A crucial role in our upcoming analysis is played by the concept of {\em epi-convergence}, which 
is now formally defined.

\begin{definition}[Epi-convergence]\label{def:EpiConv}
We say that a sequence $\{f_k\}$ of functions $f_k:\mathbb R^n\to\rbar$ {\em epi-converges} to $f:\mathbb R^n\to\rbar$ if 
$$
\Lim_{k\to \infty} \epi f_k=\epi f,
$$
where a  {\em Painlev\'e-Kuratowski} notion of set-convergence as given by \cite[Definition 4.1]{RoW 98} is employed. 
\\
In this case we  write 
$$
\elim f_k =f\quad {\rm or}\quad f_k\overset{e}{\to} f. 
$$
\end{definition}

\noindent
Epi-convergence for sequences of convex functions goes back to Wijsman
\cite{Wij 64, Wij 66}, where it is called {\em infimal convergence}. 
The term epi-convergence arguably is due to Wets \cite{Wet 80}.

\noindent
A handy characterization of epi-convergence is given by
\begin{equation}\label{eq:EpiChar}
f_k\overset{e}{\to} f\quad\Longleftrightarrow\quad\forall \bar x\in \mathbb R^n\left\{\begin{array}{ll}\forall \{x^k\}\to\bar x: & \liminf f_k(x^k)\geq f(\bar x),\\
\exists \{x^k\}\to\bar x: &  \limsup f_k(x^k)\leq f(\bar x),\end{array}\right.
\end{equation}
see \cite[Proposition 7.2]{RoW 98}, which we invoke in several places. For extensive surveys of epi-convergence we refer the reader to \cite{Att 84} or \cite[Chapter 7]{RoW 98}.
\\
\\
We make use of the {\em regular} and   {\em limiting subdifferentials}
to describe the variational behavior of nonsmooth functions.
In constructing the limiting subdifferential, we  employ the {\em outer limit} for  a set-valued mapping, which 
we now define along with the {\em inner limit}:
\\
For $S:\mathbb R^n\rightrightarrows\mathbb R^m$ and $X\subset \mathbb R^n$ 
the {\em outer limit} of $S$ at $\bar x$  {\em relative to $X$} is given by
 $$
 \Limsup_{x\to_X\bar x} S(x):=\big\{v\mid \exists \{x^k\}\to_X \bar x,  \{v^k\}\to v: v^k\in S(x^k) \quad \forall k\in\mathbb N\big\}
 $$
and the {\em inner limit} of $S$ at $\bar x$  {\em relative to $X$} is defined by
 $$
 \Liminf_{x\to_X\bar x} S(x):=\big\{v\mid \forall \{x^k\}\to_X \bar x, \,\exists \{v^k\}\to v: v^k\in S(x^k) \quad \forall k\in\mathbb N\big\}.
 $$
 We say that $S$ is {\em outer semicontinuous (osc)} at $\bar x$ {\em relative to} $X$ if 
$$
\Limsup_{x\to_X\bar x} S(x)\subset S(\bar x).
$$ 
 In case that outer and inner limit coincide, we write
 $$
 \Lim_{x\to_X \bar x} S(x):=\Limsup_{x\to_X \bar x} S(x),
 $$
and say that $S$ is {\em contiuous} at $\bar x$ {\em relative to $X$}. 

\begin{definition}[Regular and limiting subdifferential]\label{def:subdiff}
 Let  $f:\R^n\to\R\cup\{+\infty\}$ and $\bar x\in \dom f$.
\begin{itemize}
   \item[a)] The {\em regular subdifferential} of $f$ at $\bar x$ is the set 
      given by
      $$
         \hat\p f(\bar x):=\big\{v\mid f(x)\geq f(\bar x)+v^T(x-\bar x)+o(\|x-\bar x\|)\big\}.
      $$
    
   \item[b)] The {\em limiting subdifferential} of $f$ at $\bar x$ is the set 
      given by
      $$
         \p f(\bar x):=\Limsup_{x\to_f \bar x} \hat\partial f(x).
      $$
%
      
\end{itemize}
\end{definition}

\noindent
There are other ways to obtain the limiting subdifferential  
than the one described above, which goes back to Mordukhovich, 
e.g., cf. \cite{Mor 1976}. See \cite{BoL 00} 
or \cite{Iof 84} for a construction of the limiting subdifferential via 
{\em Dini-derivatives}.
\\
It is a  well-known fact, see \cite[Proposition 8.12]{RoW 98}, that if  
$f:\mathbb R^n\to\mathbb R\cup\{+\infty\}$ is convex, both the limiting and the regular
subdifferential coincide with the subdifferential of convex analysis, i.e.,
$$
\p f(\bar x)=\big\{v\mid f(x)\geq f(\bar x)+v^T(x-\bar x)\quad \forall x\in\mathbb R^n\big\}=\hat\p f(\bar x)
\quad
\forall \, \bar x\in \dom f.
$$
The above subdifferentials are closely tied to {\em normal cones}, in fact the  {\em regular } and the {\em limiting normal cone}, see \cite[Definition 6.3]{RoW 98}, of a closed set  $C\subset \mathbb R^n$ at $\bar x\in C$  can  be expressed as 
$$
\hat N(\bar x \mid C)=\hat\partial \delta (\bar x\mid C)\quad{\rm and}\quad N(\bar x\mid C)=\p \delta(\bar x\mid C),
$$
see \cite[Exercise 8.14]{RoW 98}.

\noindent
An important concept in the context of subdifferentiation is {\em (subdifferential) regularity}. We say that $f:\mathbb R^n \to\mathbb R\cup\{+\infty\}$ is
{\em (subdifferentially) regular}
at $\bar x\in \dom f$ if 
$$
N((\bar x,f(\bar x))\mid \epi f)=\hat N((\bar x,f(\bar x))\mid \epi f).
$$
\noindent
Note that this regularity notion coincides with the one used in \cite{Cla 83}, see the discussion on page 61 in \cite{Cla 83} in combination with
\cite[Corollary 6.29]{RoW 98}.

\section{Epi-Smoothing Functions}\label{sec:smooth}

\noindent
In this section we lay out the general framework for the smoothing functions 
studied in this paper. 
Let 
$f:\mathbb R^n\to \R\cup\{+\infty\}$ be lsc. We say $s_f:\mathbb R^n\times 
\mathbb R_+ \to \mathbb R$ is an {\em epi-smoothing function} for $f$ if 
the following two conditions are satisfied:
\begin{itemize}
\item[(i)] $s_f(\cdot,\mu_k)$ epi-converges to $f$ for all $\{\mu_k\}\downarrow 0$, 
written
\begin{equation}\label{eq:EpiConv}
\elim_{\mu\downarrow 0} s_f(\cdot,\mu)=f ,
\end{equation}
\item[(ii)]   $s_f(\cdot,\mu)$ is continuously differentiable for all $ \mu > 0 $. 
\end{itemize}

\noindent
Note that \eqref{eq:EpiConv}  is always fulfilled, see \cite[Theorem 7.11]{RoW 98}, under the following condition
\begin{equation}\label{eq:ContConv}
  \lim_{\mu\downarrow 0, x\to \bar x} s_f(x,\mu)=f(\bar x)\quad\forall \bar x \in\mathbb R^n,
\end{equation}
which is called {\em continuous convergence} in \cite{RoW 98}. 
As we will see in Section \ref{sec:InfConv}, however, 
continuous convergence can be an excessively strong assumption, especially when dealing with non-finite valued functions. 

\noindent
The following result provides an elementary calculus for epi-smoothing functions.

\begin{proposition}\label{prop:SmoothCalc} Let $g,h:\mathbb R^m\to\mathbb R\cup\{+\infty\}$ be lsc and let $s_g$ and $s_h$ be epi-smoothing functions for $g$ and $h$, respectively. 

\begin{itemize}
\item[a)] If $s_g$ converges continuously to $g$,  then $s_f:=s_g+s_h$ is an epi-smoothing function for $f:=g+h$.
\item[b)] If $g$ is continuously differentiable, then $s_f:=g+s_h$ is  an epi-smoothing function $f:=g+h$.
\item[c)] If $\lambda>0$, then  $\lambda s_g$ is an epi-smoothing function for $\lambda g$. 
\item[d)] If $A\in\mathbb R^{m\times n}$ has rank $m$ and $b\in \R^m$, 
then  $s_g(\cdot,\cdot):=s_g(A (\cdot)+b,\cdot)$ is  an epi-smoothing function for 
$f:=g(A(\cdot) +b)$.
\end{itemize}

\end{proposition}

\begin{proof}
Item a) follows  from \cite[Theorem 7.46]{RoW 98}, while b) follows from a) and the fact that $g$ is a continuously convergent epi-smoothing function for itself.
Item c) is provided by \cite[Exercise 7.8 d)]{RoW 98}. Item d) is an immediate consequence of Theorem \ref{th:SmoothingChain} and the discussion up front.
\end{proof}

\noindent
To obtain a more powerful chain rule than the one given in item d) above, we need to invoke more refined tools from variational analysis. One such tool is {\em metric regularity} 
(e.g., see \cite{BoL 00, Mor 06a, RoW 98}), originally defined for set-valued mappings. 
For a single-valued mapping $F:\mathbb R^n\to\R^m$ 
we say that $F$ is {\em metrically regular} at $\bar x\in\mathbb R^n$ 
if there exists $\gamma > 0$ and neighborhoods $W$ of $\bar x$
and $V$ of $F(\bar x)$ such that 
\begin{equation*}
\dist(x,F^{-1}(y))\leq \gamma \|F(x)-y\|\quad \forall x\in W, y\in V.
\end{equation*}
We  say that $F$ is metrically regular, if it is metrically regular at every 
$\bar x\in \mathbb R^n$. In particular, $F$ is metrically regular
if it is a {\em locally Lipschitz homeomorphism} (e.g., see \cite[Corollary 9.55]{RoW 98}).
Mordukhovich has shown that metric regularity can
be fully characterized via the  {\em coderivative criterion}, e.g., see \cite{Mor 06a,RoW 98}. 
In the 
case of a single-valued, continuously differentiable map  
$F:\mathbb R^n\to\mathbb R^m$ the coderivative criterion reduces to
the condition that $\rank F'(\bar x)=m$, that is,
\\
\begin{center}
$F$  is metrically regular at $\bar x$ \quad $\Longleftrightarrow$\quad  $\rank F'(\bar x)=m$.
\end{center}
\hfill
\\
\begin{theorem}\label{th:SmoothingChain}
Let $g:\mathbb R^m\to\mathbb R\cup\{+\infty\}$ and let $s_g$ be an epi-smoothing function  for $g$. Furthermore, let $F:\mathbb R^n\to\mathbb R^m$ be continuously differentiable and metrically regular. Then $s_f:=s_g(F(\cdot),\cdot)$ is an epi-smoothing function for $f:=g\circ F$. 
\end{theorem}
\begin{proof}  
The smoothness properties are obvious from the assumptions. 
Next, let $\{\mu_k\}\downarrow $ be given and put $g_k:=s_g(\cdot,\mu_k)$ and $f_k:=g_k\circ F$. We need to show that  $f_k\overset{e}{\to} f$. For this purpose, we invoke the characterization of epi-convergence as provided by 
\eqref{eq:EpiChar}.  To this end, let $\bar x\in\mathbb R^n$ and $\{x^k\}\to\bar x$ be given. Then it follows from the fact that $g_k\overset{e}{\to} g$ and 
\eqref{eq:EpiChar} 
that 
\begin{equation}\label{eq:Aux1}
\liminf_k f_k(x^k)= \liminf_k g_k(F(x^k))\geq g(F(\bar x)) = f(\bar x).
\end{equation}

\noindent
Moreover, as $g_k\overset{e}{\to} g$, \eqref{eq:EpiChar} yields a sequence $\{y^k\}\to \bar y:=F(\bar x)$ such that 
$$
\limsup_k g_k(y^k)\leq g(\bar y).
$$
Since $F$ is metrically regular at $\bar x$, we obtain a sequence $\{x^k\}\to\bar x$ such that $F(x^k)=y^k$ for all $k\in\mathbb N$. This, in turn, gives
$$
\limsup_k f_k(x^k)=\limsup_k g_k (y^k)\geq \bar y = f(\bar x).
$$ 
This, together with \eqref{eq:Aux1} proves \eqref{eq:EpiChar} for $f_k$ with respect to $f$, and this concludes the proof.
\end{proof}

\noindent
Although epi-convergence is arguably a mild condition, it still provides desirable convergence  behavior for minimization in the following sense:

\begin{theorem}\cite[Theorem 7.33]{RoW 98} \label{th:EpiInf} Suppose the sequence $\{f_k\}$ is {\em eventually level-bounded} (see \cite[p. 266]{RoW 98}), and $f_k\overset{e}{\to}f$ with 
$f_k$ and $f$ lsc and proper. Then 
$$
\inf f_k\to\inf f \quad (finite).
$$
\end{theorem}

\noindent
Now, suppose a numerical algorithm produces sequences  $\{x^k\}\to \bar x$ and $\{\mu_k\}\downarrow 0$ such that 
$$
\lim_{k\to \infty} \nabla_x s_f(x^k,\mu_k)\to 0.
$$ 
A natural question to ask in this context is whether $\bar x$ is a critical point of $f$ in the sense that $0\in \p f(\bar x)$. 
A sufficient condition is, clearly, provided  by 
$$
\Limsup_{x\to\bar x,\mu\downarrow 0} \nabla_x s_f(x,\mu)\subset \partial f(\bar x).
$$

\noindent
The next result shows that the converse inclusion is always valid  if $s_f(\cdot,\mu)\overset{e}{\to}f$.

\begin{lemma}\label{lem:Gincl} 
Let $f:\mathbb R^n\to\R\cup\{+\infty\}$ be lsc and $s_f$ an epi-smoothing function for $f$. 
Then for $\bar x\in \dom f$ we have 
$$
   \p f(\bar x)\subset \Limsup_{x\to \bar x, \mu\downarrow 0}\nabla_xs_f(x,\mu).
$$
\end{lemma}

\begin{proof} 
Let $v\in \p f(\bar x)$ be given. Since by assumption $\elim_{\mu\downarrow 0} s_f(\cdot,\mu)=f$
we may  invoke  \cite[Corollary 8.47]{RoW 98} in order to obtain sequences $\{\mu_k\}\downarrow 
0, \{x^k\}\to\bar x$ and $\{v^k\}$ with $v^k\in\partial_x s_f(x^k,\mu_k) $
such that $v^k\to v$. Now, since $s_f(\cdot,\mu_k)$ is continuously
differentiable by assumption, we have 
$$
   v^k=\nabla_x f(x^k,\mu_k),
$$
which identifies $v$ as an element of $\Limsup_{x\to \bar x, \mu\downarrow 0}\nabla_xs_f(x,\mu)$ and thus, 
the assertion follows.
\end{proof}

\noindent
A major contribution of this paper is the construction of smoothing functions
having the property that
\begin{equation}\label{eq:gc}
\Limsup_{x\to\bar x,\mu\downarrow 0} \nabla_x s_f(x,\mu) = \partial f(\bar x)
\end{equation}
at any point $\bar x\in \dom f$. This condition implies the notion of gradient 
consistency defined in \cite[Equation (4)]{Che 12} which is obtained by taking the convex
hull on both sides of this equation. However, since all of the functions we
consider are subdifferentially regular, Lemma \ref{lem:Gincl} implies that
\eqref{eq:gc} is equivalent to gradient consistency.

\section{Epi-Smoothing via Infimal Convolution}\label{sec:InfConv}

\noindent
In this section we show that the class of smoothing functions for nonsmooth, convex and lsc functions introduced in \cite{BeT 12} fits into the 
framework layed out in  Section \ref{sec:smooth}. As a by-product, we show that  
{\em Moreau envelopes} fulfill the 
requirements of our smoothing setup.
\\
The approach taken in \cite{BeT 12} is based on {\em infimal convolution} 
\cite{BaC 11, HUL 96, HUL 01,Roc 70,RoW 98}.
Given two (extended real-valued) functions  $f_1,f_2:\mathbb R^n\to \rbar $  the  {\em inf-convolution} (or {\em epi-sum}, see Lemma \ref{lem:AuxEpi} b) in this context) is the function
$f_1\# f_2:\mathbb R^n\to \rbar $ defined by 
\begin{equation*}\label{eq:InfConv}
(f_1\#f_2)(x):=\inf_{u\in\mathbb R^n} \{f_1(u)+f_2(x-u)\}.
\end{equation*}
\noindent
In what follows we assume that
\begin{enumerate}
\item[(A)] $g:\mathbb R^n\to\mathbb R\cup\{+\infty\}$ is 
proper, lsc, and convex, and 
\item[(B)]
$\omega:\mathbb R^n\to\mathbb R$ is 
convex and continuously differentiable with Lipschitz gradient.
\end{enumerate}
Moreover, for $\mu>0$, define the function 
$\omega_\mu:\mathbb R^n\to\mathbb R\cup\{+\infty\}$ by 
$$
\omega_\mu(y):=\mu\omega\Big(\frac{y}{\mu}\Big).
$$
Obviously, $\omega_\mu$ is also convex and continuously differentiable with Lipschitz gradient.
\\
In \cite{BeT 12}, the authors consider the (convex) function 
$$
(g\#\omega_\mu)(x)=\inf_{u\in\mathbb R^n} \Big\{g(u)+\mu \omega\Big(\frac{x-u}{\mu}\Big)\Big\}\quad (\mu>0)
$$
as a smoothing function for $g$. 
We  now investigate  conditions on $\omega$ for 
which  the inf-convolution $g\#\omega_\mu$ serves as an 
epi-smoothing function in the sense of Section \ref{sec:smooth}.
In this context, the notion of {\em coercivity} plays a key role where it 
arises as a natural assumption on the function $\omega$.
Several different notions of coercivity occur in the literature.
We now define those useful to our study.

\begin{definition}[Coercive functions]\label{def:Coercive} Let $f:\mathbb R^n\to\mathbb R\cup\{+\infty\}$ be lsc and convex.
\begin{itemize}
 \item[a)]  $f$ is called {\em 0-coercive} if 
$$
\lim_{\|x\|\to\infty} f(x)=+\infty.
$$
\item[b)] $f$ is called {\em 1-coercive} if 
$$
\lim_{\|x\|\to\infty} \frac{f(x)}{\|x\|}=+\infty.
$$
\end{itemize}
\end{definition}

\noindent
The first result establishes important properties of the function $g\#\omega_\mu$.

\begin{lemma}\label{lem:AuxEpi} 
If $\omega$ is 1-coercive (or 0-coercive and $g$ bounded from below) the following holds:
\begin{itemize}
\item[a)] $g\#\omega_\mu$ is finite-valued, i.e., $g\#\omega_\mu:\mathbb R^n\to\mathbb R$,
and for all $x\in\mathbb R^n$ we have 
$$
(g\#\omega_\mu)(x)=\min_{u\in\mathbb R^n} \Big\{g(u)+\mu \omega\Big(\frac{x-u}{\mu}\Big)\Big\}
$$
i.e., 
$$
\argmin_{u\in\mathbb R^n} \Big\{g(u)+\mu \omega\Big(\frac{x-u}{\mu}\Big)\Big\} \neq \emptyset.
$$
\item[b)] We have
$$
\epi g\#\omega_\mu= \epi g +\epi \omega_\mu.
$$
\item[c)] $g\#\omega_\mu$ is continuously differentiable with 
$$
\nabla(g\#\omega_\mu)(x)=\nabla \omega\Big(\frac{x-u_\mu(x)}{\mu}\Big)=\nabla \omega_\mu(x-u_\mu(x))\quad\forall x\in\mathbb R^n,
$$
where $u_\mu(x)\in\argmin_{u\in\mathbb R^n} \Big\{g(u)+\mu \omega\Big(\frac{x-u}{\mu}\Big)\Big\}$.
\end{itemize}
\end{lemma}

\begin{proof} The assertion that 
$$
(g\#\omega_\mu)(x)<+\infty\quad\forall x\in\mathbb R^n
$$
is due to the fact that $\omega$ is finite-valued and $g\not\equiv +\infty$. 
Moreover,  $\omega_\mu$ obviously inherits the respective coercivity properties from $\omega$.
Hence, the remainder of  a) follows immediately from \cite[Proposition 12.14]{BaC 11}.
\\
In turn, b) follows from a) and \cite[Proposition 12.8 (ii)]{BaC 11}.
\\
Item c) is an immediate consequence of a) together with \cite[Theorem 4.2 (c)]{BeT 12}.
\end{proof}

\noindent
The following auxiliary result, which is key for establishing epigraphical limit behavior of $g\#\omega_\mu$, states that  the epigraphical limit of $\omega_\mu$ for $\mu\downarrow 0$ is $\delta(\cdot\mid \{0\})$  if and only if  $\omega$ is 1-coercive.

\begin{lemma}\label{lem:EpiOm}   $\omega$ is 1-coercive if and only if 
$$
\elim_{\mu\downarrow 0} \omega_\mu= \delta (\cdot\mid \{0\}).
$$
 \end{lemma}
\begin{proof}
First,  let $\omega$ be 1-coercive:\\
We start by showing  that 
$\Limsup_{\mu\downarrow 0}\epi \omega_\mu\subset\{0\}\times \mathbb R_+=\epi \delta (\cdot\mid \{0\})$.
\\
To this end, let $(\bar z,\bar \alpha)\in \Limsup_{\mu\downarrow 0}\epi \omega_\mu$. Then there exist sequences $\{z^k\}\to\bar z$, $\{\alpha_k\}\to\bar \alpha$ and 
$\{\mu_k\}\downarrow 0$ such that
\begin{equation}\label{eq:EpiIneq}
\mu_k\omega\Big(\frac{z^k}{\mu_k}\Big)\leq \alpha_k\quad \forall k\in\mathbb N.
\end{equation}
This can be written as 
$$
\omega\Big(\frac{z^k}{\mu_k}\Big)\leq \frac{\alpha_k}{\mu_k}\quad \forall k\in\mathbb N.
$$
It is immediately clear from this representation, that $\bar\alpha\geq 0$, since otherwise the right-hand side would tend to $-\infty$, while the left-hand
side remains either convergent on a subsequence (if $\{\frac{z^k}{\mu_k}\}$ is bounded) or tends to $+\infty$ (if $\{\frac{z^k}{\mu_k}\}$ is unbounded).
\\
Now, suppose that  $\bar z\neq 0$. Then  $\{\frac{z^k}{\mu_k}\}$ is unbounded and \eqref{eq:EpiIneq} can be rewritten as 
$$
\frac{\omega\Big(\frac{z^k}{\mu_k}\Big)}{\|\frac{z^k}{\mu^k}\|}\leq \frac{\alpha_k}{\|z^k\|}\quad \forall k\in\mathbb N.
$$
By the 1-coercivity of $\omega$ the left-hand side tends to $+\infty$, while the right-hand side is bounded, which is a contradiction. Hence, we have proven that 
$\bar z=0$ and $\bar \alpha\geq 0$, which shows that, in fact, $\Limsup_{\mu\downarrow 0}\epi \omega_\mu\subset\{0\}\times \mathbb R_+$.
\\
We now show that $\Liminf_{\mu\downarrow 0}\epi \omega_\mu\supseteq\{0\}\times \mathbb R_+$. For these purposes, let $\bar \alpha\geq 0$ and $\{\mu_k\}\downarrow 0$ be given. 
Then choose $z^k:=0$ and $\alpha_k:=\bar\alpha+\mu_k\omega(0)\geq \omega_{\mu_k}(z^k)$. Then $(z^k,\alpha_k)\in \epi \omega_{\mu_k}$ for all $k\in\mathbb N$ and
$(z^k,\alpha_k)\to(0,\bar\alpha)$. This shows that $\Liminf_{\mu\downarrow 0}\epi \omega_\mu\supseteq\{0\}\times \mathbb R_+$. 
\\
Putting together all the pieces of information, we see that 
$$
\Lim_{\mu\downarrow 0}\epi \omega_\mu= \epi \delta(\cdot\mid \{0\}),
$$
i.e.,
$$
\elim_{\mu\downarrow 0} \omega_\mu =  \delta(\cdot\mid \{0\}).
$$

Now, suppose that $\omega$ is not 1-coercive. Then there exists an unbounded sequence $\{x^k\}$ such that either 
$$
\frac{\omega(x^k)}{\|x^k\|}\to-\infty
$$
or $\Big\{\frac{\omega(x^k)}{\|x^k\|}\Big\}$ is bounded. Put $\mu_k:=\frac{1}{\|x^k\|}\to 0$. Then 
$$
\omega_{\mu_k}\Big(\frac{x^k}{\|x^k\|}\Big) =\frac{\omega(x^k)}{\|x^k\|},
$$
and we have 
\begin{equation}\label{eq:EpiEq}
\Big(\frac{x^k}{\|x^k\|}, \omega_{\mu_k}\Big(\frac{x^k}{\|x^k\|}\Big)\Big)\in \epi \omega_{\mu_k}\quad\forall k\in\mathbb N.
\end{equation}
If $\frac{\omega(x^k)}{\|x^k\|}\to-\infty$, we infer that $\omega_{u_k}$ does not converge epigraphically at all (in particular not to $\delta(\cdot\mid \{0\})$)
from \eqref{eq:EpiChar}, since we  have $\liminf_{k\to\infty} \omega_{\mu_k}\Big(\frac{x^k}{\|x^k\|}\big)\to-\infty$.
\\
In case that $\Big\{\frac{\omega(x^k)}{\|x^k\|}\Big\}$ is bounded, we may assume w.l.g. that 
$$
\frac{\omega(x^k)}{\|x^k\|}\to\bar\omega
$$
for some $\bar \omega\in \mathbb R$. Then we infer from \eqref{eq:EpiEq} that 
$$
(\bar x,\bar \omega)\in \Limsup_{k\to\infty}\epi \omega_{\mu_k}
$$
with $\bar x\neq 0$ being an accumulation point of $\Big\{\frac{x^k}{\|x^k\|}\Big\}$. But $(\bar x,\bar \omega)\notin\epi \delta (\cdot\mid \{0\})$, which concludes the proof.
\end{proof} 

\noindent
The following lemma establishes monotonicity properties for the family of functions
$g\#\omega_\mu$, which come into play in Section \ref{sec:ConvCom}.

\begin{lemma}\label{lem:Mono} If $\omega(0)\leq 0$, then 
for all $x\in\mathbb R^n$ the function $\mu\mapsto (g\#\omega_\mu)(x)$ is nondecreasing on  $\mathbb R_{++}$ and bounded by $g(x)$ from above.
\end{lemma}
\begin{proof}
 Let $y\in\mathbb R^n$. Then for $\mu_1>\mu_2>0$ we have 
\begin{eqnarray*}
\omega\Big(\frac{y}{\mu_1}\Big) & = & \omega\Big(\frac{\mu_2}{\mu_1}\frac{y}{\mu_2}+\Big(1-\frac{\mu_2}{\mu_1}\Big)0\Big)\\
&\leq & \frac{\mu_2}{\mu_1}\omega\Big(\frac{y}{\mu_2}\Big)+\Big(1-\frac{\mu_2}{\mu_1}\Big)\omega(0)\\
& \leq &  \frac{\mu_2}{\mu_1}\omega\Big(\frac{y}{\mu_2}\Big).
\end{eqnarray*}
Multiplying by $\mu_1$ yields 
$$
\omega_{\mu_1}(y)\leq \omega_{\mu_2}(y)\quad\forall y\in\mathbb R^n,
$$
and hence for $x\in\mathbb R^n$ arbitrarily given, we have
$$
g(u)+\omega_{\mu_1}(x-u)\leq  g(u)+\omega_{\mu_2}(x-u)\quad\forall u\in\mathbb R^n.
$$
Taking the infimum over all $u\in\mathbb R^n$ gives 
$$
(g\#\omega_{\mu_1})(x)\leq (g\#\omega_{\mu_2})(x),
$$
which concludes the proof due  the choice of $\mu_1$ and $\mu_2$.
\end{proof}

\noindent
The following result establishes the desired 
epi-convergence properties of the inf-convolutions. Note  that, to our knowledge, we cannot deduce it
from known results such as  \cite[Proposition 7.56]{RoW 98} or \cite[Theorem 4.2]{AtW 89}, 
since our assumptions do not meet the requirements for the application of these results.
In particular, we do not assume $g$ to be bounded from below.

\begin{proposition}\label{prop:EpiLim}
If   $\omega$   is 1-coercive, then
$$
\elim_{\mu\downarrow 0} g\#\omega_\mu = g.
$$
\end{proposition}

\begin{proof} The fact that $\Liminf_{\mu\downarrow 0} \epi g\#\omega_\mu\supseteq \epi g$ follows immediately from \cite[Theorem 4.29 a)]{RoW 98} when applied to the respective epigraphs.
\\
Therefore, it is enough to show that $\Limsup_{\mu\downarrow 0} \epi g\#\omega_\mu\subset \epi g$.
\\
To this end, pick $(\bar x,\bar \alpha)\in \Limsup_{\mu\downarrow 0} \epi g\#\omega_\mu$ arbitrarily. Then there exist sequences $\{\mu_k\}\downarrow 0, \{x^k\}\to\bar x$ and $\alpha_k\to\bar \alpha$ such that 
\begin{equation}\label{eq:Epi1}
 (g\#\omega_{\mu_k})(x^k)\leq \alpha_k\quad \forall k\in\mathbb N.
\end{equation}
With 
$$
u^k\in\argmin_{u\in\mathbb R^n} \Big\{g(u)+\mu_k\omega\Big(\frac{x^k-u}{\mu_k}\Big)\Big\},
$$
 \eqref{eq:Epi1} can be written as
\begin{equation}\label{eq:Epi2}
g(u^k)+\mu_k\omega\Big(\frac{x^k-u^k}{\mu_k}\Big)\leq \alpha_k\quad \forall k\in\mathbb N.
\end{equation}
Using the fact, cf.  \cite[Theorem 9.19]{BaC 11}, that  the convex, lsc function $g$ is minorized by  an affine function, say  $x\mapsto b^Tx+\beta$, 
this leads to 
$$
b^Tu^k+\beta+\mu_k\omega\Big(\frac{x^k-u^k}{\mu_k}\Big)\leq \alpha_k\quad \forall k\in\mathbb N.
$$
If we assume that $\{u^k\}$ does not convergence to $\bar x$, we can rewrite this (for $k$ sufficiently large) as
$$
\frac{\omega\Big(\frac{x^k-u^k}{\mu_k}\Big)}{\|\frac{x^k-u^k}{\mu_k}\|}\leq \frac{\alpha_k-b^Tu^k-\beta}{\|x^k-u^k\|}.
$$ 
Whether $\{u^k\}$ is unbounded or not, we obtain a contradiction, since  the left-hand side tends to $+\infty$, as $\omega$ is 1-coercive, while the right-hand side
remains bounded.

Hence, $\{u^k\}\to\bar x$. We now claim that  $g(u^k)\not\to+\infty$, and hence, in particular,   $\bar x\in\dom g$. If this were not the case, we  invoke  \cite[Theorem 9.19]{BaC 11} again  to get an affine minorant of $\omega$, say $x\mapsto c^Tx+\gamma$, and infer from \eqref{eq:Epi2} that
$$
g(u^k)+c^T(u^k-x^k)+\mu_k\gamma\leq \alpha_k\quad \forall k\in \mathbb N.
$$
This, however, leads to a contradiction if $g(u^k)\to +\infty$ since $c^T(u^k-x^k)+\mu_k\gamma\to 0$ and $\alpha_k\to\bar \alpha<+\infty$. Thus, we have 
shown that $\{g(u^k)\}$ is bounded from above. Since  $g$ is lsc and $u^k\to\bar x$, we also know that $\liminf_{k\to \infty}g(u^k)\geq g(\bar x)$.
Hence, we may as well assume that $g(u^k)\to \hat g\geq g(\bar x)$ and, in particular, we have $\bar x\in\dom g$. 
\\
We now infer from \eqref{eq:Epi2} that 
$$
(x^k-u^k, \alpha_k-g(u^k))\in \epi \omega_{\mu_k}\quad \forall k\in\mathbb N.
$$
Since $x^k-u^k\to 0$ and $\alpha_k-g(u^k)\to \alpha-\hat g$,  Lemma \ref{lem:EpiOm} implies
$$
(0,\bar \alpha-\hat g)\in \Limsup_{\mu\downarrow 0}\epi \omega_\mu\subset\epi \delta (\cdot\mid \{0\}).
$$
This immediately gives 
$$
g(\bar x)\leq \hat g\leq \bar \alpha,
$$ 
i.e., $(\bar x,\bar \alpha)\in\epi g$, which concludes the proof.
\end{proof}

\noindent
We are now in a position to state the main result of this section.

\begin{theorem}\label{th:InfCon}
 If  $\omega$ is 1-coercive then the function $s_g:(x,\mu)\mapsto (g\#\omega_\mu)(x)$ is an 
 epi-smoothing function for $g$ with 
$$
\gph \nabla_x s_g(\cdot,\mu)\underset{\mu\downarrow 0}{\rightarrow}\gph \partial g,
$$
and hence, in particular,
$$
\Limsup_{\mu\downarrow 0,\, x\to\bar x}  \nabla_xs_g(x,\mu)=\partial g(\bar x)\quad \forall \bar x\in\dom g.
$$
\end{theorem}

\begin{proof}
Due to  Propostion \ref{prop:EpiLim}, we have $\elim_{\mu\downarrow 0} s_g(\cdot,\mu)=\elim_{\mu\downarrow 0} g\#\omega_\mu = g$. The smoothness properties of $\nabla_xs_g(\cdot,\mu)=\nabla g\#\omega_\mu$ follow from
Lemma \ref{lem:AuxEpi}. The remaining assertion is an immediate consequence of {\em Attouch's Theorem}, see
\cite[Theorem 12.35]{RoW 98}. This concludes the proof.
\end{proof}

 \subsection*{Moreau Envelopes} The most prominent choice  for $\omega$ is given by
$$
\omega:=\frac{1}{2}\|\cdot\|^2.
$$
\noindent
The resulting inf-convolution of $\omega_\mu$ with an lsc  function $g:\mathbb R^n\to\mathbb R\cup\{+\infty\}$  is called the {\em Moreau envelope} or {\em Moreau-Yosida regularization} of $g$ and is denoted by  $e_\mu g$, i.e., 
    $$
          e_\mu g(x)=\inf_w\Big\{g(w)+ \frac{1}{2\mu}\|w-x\|^2\Big\}.
       $$
 The set-valued map $P_\mu g:\mathbb R^n\rightrightarrows \mathbb R^n $  
      given by 
      $$
         P_\mu g(x):=\argmin_w\Big\{g(w)+ \frac{1}{2\mu}\|w-x\|^2\Big\}
      $$
      is called the {\em proximal mapping} for $g$.

\noindent
The following properties of Moreau envelopes and proximal mappings of convex 
functions are well known, see \cite{Roc 70,RoW 98} or \cite{HUL 96}.

\begin{proposition}\label{moreau basics}
Let $f:\mathbb R^n\to\mathbb R\cup\{+\infty\}$ be lsc and convex and $\mu >0$. Then the 
following holds:
\begin{itemize}
   \item[a)] $P_\mu f$ is single-valued and Lipschitz continuous.
   \item[b)] $e_\mu f$ is  convex and smooth with Lipschitz gradient 
      $ \nabla e_\mu f$ given by
      $$
         \nabla e_\mu f(x)=\frac{1}{\mu}[x-P_\mu f(x)].
      $$
   \item[c)] $\argmin f=\argmin e_\mu f$.
\end{itemize}
\end{proposition}

\noindent
In view of item c) it is possible to recover the minimzers of a (possibly 
nonsmooth) convex function by those of its Moreau envelope. Hence,
it is not even necessary to drive the smoothing parameter to zero. 
\\
Since the function $x\mapsto\frac{1}{2}\|x\|^2$ is 1-coercive, the following result can be formulated as a corollary of Theorem \ref{th:InfCon}.

\begin{corollary}\label{cor:Moreau}
Let $g:\mathbb R^n\to\mathbb R\cup\{+\infty\}$ be  lsc and convex. Then $s_g:(x,\mu)\mapsto e_\mu g(x)$ is an epi-smoothing function for $g$ with 
$$
   \Limsup_{\mu\downarrow 0, \, x\to\bar x} \nabla_x s_g(x,\mu) = \partial g(\bar x)\quad \forall \bar x\in \dom g.
$$
\end{corollary}
When $g$ is lsc and convex, the fact that $e_\mu g$ epi-converges to $g$ as 
$\mu\downarrow 0$ is well known 
(cf. the discussion in \cite{RoW 98} after Proposition 7.4).

 \subsection*{Extended Piecewise Linear-Quadratic Functions (EPLQ) \cite{RoW 98}}
 EPLQ functions play a key role in a
 wide variety of applications, 
 e.g., 
 {\em signal denoising} \cite{CoW 12,Donoho1992}, 
 {\em model selection} \cite{Lasso1996},
 {\em compressed sensing} \cite{Donoho2006,LARS2004,Hastie90}, 
 {\em robust statistics} \cite{Hub 81}, 
{\em  Kalman filtering} \cite{Aravkin2011,AravkinIFAC,Farahmand2011}, and 
 {\em support vector classifiers} \cite{Evgeniou99,Pontil98,Scholkopf00}. 
 Examples include arbitrary gauge functionals 
 \cite{RoW 98} (e.g., norms), 
 the 
{\em Huber penalty} \cite{BeT 12,Hub 81}, 
 the 
 {\em hinge loss function} \cite{Evgeniou99,Pontil98,Scholkopf00}, 
 and the 
 {\em Vapnik penalty} \cite{Hastie01,Vapnik98}.
 For an overview of these functions and their statistical properties see
 \cite{ABP 11,RoW 98}.
 In this section, we show that the Moreau envelope mapping $g\mapsto e_\mu g$ maps the class
 of EPLQ functions to itself in a very natural way.
 
 \begin{definition} 
 The convex function \(\map{g}{\bR^n}{\eR}\) is said to be 
 extended piecewise linear-quadratic
 if for some positive integer \(m\) there exists a nonempty closed convex
 set \(U\subset\bR^m\) (typically polyhedral), an injective matrix $R\in\bR^{n\times m}$, 
 a symmetric and positive semi-definite matrix \(B\in\bR^{m\times m}\),
 and a vector $b\in\bR^m$ such that
 \begin{equation}\label{eplq}
 g(x):=\theta_\eplqs(x):=\sup_{u\in U}\ip{u}{Rx-b}-\half u^TBu .
 \end{equation}
 If \(m=n\), $R=I$, and $b=0$, then $g$ is said to be {\em piecewise linear-quadratic}
 (PLQ).
 \end{definition}
 
 \begin{example}[Examples of EPLQ functions]
 
 \noindent
 \begin{enumerate}
 \item Norms: Let $\snorm{\cdot}$ be a norm with closed unit ball $\bB_*$.
 Then $\snorm{\cdot}=\theta_{\scriptscriptstyle{(\bB_*^\circ,0,I,0)}}$, where
 $\bB_*^\circ:=\set{v}{\ip{v}{u}\le 1\ \forall \, u\in\bB_*}$.
 \item The Huber penalty: Let $\kappa >0$. Then
 $\theta_{{\scriptscriptstyle{([-\kappa,\kappa]^n,I,I,0)}}}$ is the
 Huber penalty with threshold $\kappa$.
 \item The Vapnik penalty: Let $\epsilon >0$ and define
 $U=[0,1]^{2n}, R=[I_{n\times n},\, -I_{n\times n}]^T,$ and 
 $b=\epsilon\, \mathsf{ones}(2n,1)$, then
 $\theta_{{\scriptscriptstyle{(U,0,T,b)}}}$ is the Vapnik penalty with
 threshold $\epsilon$.
 \end{enumerate}
 \end{example}
 
 
 \begin{proposition}
 Let $\theta_\eplqs$ 
 be an extended piecewise linear-quadratic function.
 If $B$ is positive definite or
 $U$ is bounded, then
 \[
 e_\mu \theta_\eplqs=\theta_{\scriptscriptstyle{(U,\hat{B},R,b)}},
 \]
 where $\hat{B}=B+\mu RR^T$.
Moreover, for each $x\in\bR^n$ there exists a saddle-point \( (\bu,\bv)\in U\times \bR^n\)
for the closed proper concave-convex saddle-function \cite[Section 33]{Roc 70}
\[
K(u,v):=\ip{Rv-b}{u}-\half u^TBu+\frac{1}{2\mu}\tnorm{x-v}^2-\delta(u\mid U)
\]
satisfying \(e_\mu g(x)=K(\bu,\bv)\).
 \end{proposition}
 
 \begin{proof}
 Regardless of the choice of $x$, 
 $K$ is coercive  in $v$ for each $u\in U$, and if 
 $B$ is positive definite or
 $U$ is bounded, then 
 $-K$ is coercive in $u$ for each $v\in \bR^n$. 
 Hence, by \cite[Theorem 37.6]{Roc 70}, for every $x\in\bR^n$, $K$ has a 
 saddle-point \( (\bu,\bv)\in U\times \bR^n\) satisfying
 \begin{eqnarray*}
 e_\mu g(x)&=&\inf_{v\in\bR^n}\sup_{{u\in U}}K(u,v)\\
 &=&K(\bu,\bv)\\
 &=&\sup_{{u\in U}}\inf_{v\in\bR^n}K(u,v) .
 \end{eqnarray*}
 To complete the proof observe that the problem
 \[
 \inf_{v\in\bR^n}K(u,v)=-\left[\ip{b}{u}+\half u^TBu\right]+
 \inf_{v\in\bR^n}\left[\ip{v}{R^Tu}+\frac{1}{2\mu}\tnorm{x-v}^2\right]
 \]
 has a unique solution at $v(x,u)=x-\mu R^Tu$. Plugging this solution into $K$ gives
 \(
 e_\mu g(x)=\sup_{u\in U}K(u,v(x,u))=\theta_{\scriptscriptstyle{(U,\hat{B},R,b)}}(x).
 \)
 \end{proof}

\begin{example}[Lasso-Problem]\label{ex:Lasso}
Given $A \in\mathbb R^{m\times n}$ and $b\in\mathbb R^m$ 
with $m<<n$,  consider the nonsmooth optimization problem
\begin{equation}\label{eq:Lasso}
   \min_x f(x):=\frac{1}{2}\|Ax-b\|^2+\lambda \|x\|_1,
\end{equation}
where $\lambda>0$. This problem is known in the literature as the 
{\em Lasso-Problem}, see \cite{LARS2004,Lasso1996}. 

The objective function $f$ is the sum of two convex functions,
one is smooth and the other is a nonsmooth PLQ function. 
By Proposition \ref{prop:SmoothCalc},
an epi-smoothing function for $f$ can be obtained by computing the
Moreau envelope for the $1$-norm. This envelope is obtained from
the proximal mapping which in this case is commonly referred to 
in the literature as
\emph{soft thresholding} \cite{CoW 12,Donoho1992}.
An easy 
computation shows that 
$$
   P_\mu \|\cdot\|_1(x)=\left\{\begin{array}{ccl}
   x_i+\mu & {\rm if} & x_i<-\mu,\\
   x_i-\mu & {\rm if} & x_i>\mu,\\
   0 & {if} & |x_i|\leq \mu.
   \end{array}\right.
$$
\end{example}

\section{Convex Composite Functions}\label{sec:ConvCom}

\noindent 
An important and powerful class  of nonsmooth, nonconvex functions   $f:\mathbb R^n\to\mathbb R\cup\{+\infty\}$ is given by  
\begin{equation}\label{eq:ConvCom}
 f(x):=g(H(x))\quad\forall x\in\mathbb R^n,
\end{equation}
 where $g:\mathbb R^m\to\mathbb R\cup\{+\infty\}$  is lsc and convex and $H:\mathbb R^n\to\mathbb R^m$  (twice) continuously differentiable. These functions go by the name {\em convex composite}, see, e.g., \cite{Bur 85, Bur 87} or \cite{BuP 92}, and  are  closely related to  {\em amenable} functions, see \cite[Definition 10.32]{RoW 98}.
\\
Suppose one has an epi-smoothing function $s_g$ of $g$, then it is a natural 
question to ask whether $s_f(\cdot,\cdot):=s_g(H(\cdot),\cdot)$ is 
an epi-smoothing function of $f$. That is, do the smoothing properties of $s_g$ 
(with respect to $g$)  carry over to smoothing properties of $s_f$ (with respect to $f$)?  
In particular, does the 
epi-convergence of $s_g(\cdot,\mu)$ to $g$ imply the epi-convergence
of  $s_f(\cdot,\mu)$ to $f$\;? To clarify this connection, 
we start with an easy observation for which we give a self-contained proof
(an alternative proof can be obtained by applying \cite[Formula 4(8)]{RoW 98} to the respective epigraphs and the function 
$F(x,\alpha):=(H(x),\alpha)$ satisfying 
$\epi f=F^{-1}(\epi g)$).

\begin{lemma}\label{lem:fEpiLimSup}
 Let  $s_g$ be an epi-smoothing function for $g$, and define  
 $s_f(\cdot,\cdot):=s_g(H(\cdot),\cdot)$. Then
$$
\Limsup_{\mu \downarrow 0}\epi s_f(\cdot,\mu)\subset \epi f.
$$
\end{lemma}
\begin{proof} 
 Let $(\bar x,\bar \alpha)\in\Limsup_{\mu \downarrow 0}\epi s_f(\cdot,\mu)$. Then there exist sequences $\{x^k\}\to\bar x, \{\alpha_k\}\to\bar \alpha$ and $\{\mu_k\}\downarrow 0$ such that 
$$
s_g(H(x^k),\mu_k)\leq \alpha_k\quad \forall k\in\mathbb N,
$$
i.e.,
$$
(H(x^k),\alpha_k)\in\epi s_g(\cdot,\mu_k)\quad \forall k\in\mathbb N.
$$
Since $(H(x^k),\alpha^k)\to (H(\bar x),\bar \alpha)$ we get from the epi-convergence of $s_g(\cdot,\mu)$ to $g$ that
$$
(H(\bar x),\bar \alpha)\in\epi g,
$$
which immediately yields 
$$
(\bar x,\bar \alpha)\in\epi f.
$$
This proves the result.
\end{proof}

\noindent
We  point out that in the previous result, as well as in the following two results, only continuity of $H$ and no smoothness assumption is needed. 

\begin{proposition}\label{prop:fEpiLim}
Let  $s_g$ be an epi-smoothing function for $g$ such that $s_g(y,\mu) $  is nondecreasing as $ \mu\downarrow 0$  for all $y\in\mathbb R^m$. Then for 
$s_f(\cdot,\cdot):=s_g(H(\cdot),\cdot)$ we have 
$$
\elim_{\mu_\downarrow 0} s_f(\cdot,\mu) =f.
$$

\end{proposition}

\begin {proof}
Due to Lemma \ref{lem:fEpiLimSup}, it suffices to show that 
$$
\Liminf_{\mu \downarrow 0}\epi s_f(\cdot,\mu)\supseteq \epi f.
$$
To this end, let $(\bar x,\bar \alpha)\in \epi f$, i.e., $g(H(\bar x))\leq \bar \alpha$. Now, let $\{\mu_k\}\downarrow 0$ be given. In view of the monotonicity assumption
we get $s_g(H(\bar x),\mu_k)\leq \bar \alpha$ and hence
$$
(\bar x, \bar \alpha)\in \epi s_f(\cdot,\mu_k)\quad \forall k\in\mathbb N.
$$
With the choice $x^k:=\bar x$ and $\alpha_k:=\bar \alpha$ it follows immediately that 
$$
(\bar x,\bar \alpha)\in\Liminf_{\mu \downarrow 0}\epi s_f(\cdot,\mu),
$$
which concludes the proof.
\end {proof}

\begin{corollary} \label{cor:smooth}
If, in the setting of Section \ref{sec:InfConv}, $\omega$ is 1-coercive with $\omega(0)\leq 0$, then for  $s_g(\cdot,\mu):=g\#\omega_\mu$ 
we have
$$
\elim_{\mu\downarrow 0}s_g(H(\cdot),\mu) =g\circ H.
$$
\end{corollary}
\begin{proof}
 The assertion follows immediately from Lemma \ref{lem:Mono} and Proposition \ref{prop:fEpiLim}.
\end{proof}

\noindent 
In the following result we employ the  limiting normal cone for a (nonempty) convex set $C\subset \mathbb R^n$ at $\bar x\in C$, which is given by, cf. \cite[Theorem 6.9]{RoW 98},
$$
N(\bar x\mid C)=\{v\in\mathbb R^n\mid v^T(x-\bar x)\leq 0\quad \forall x\in C\}.
$$
In our setting, $C$ is the domain of an lsc, convex function $g:\mathbb R^n\to\mathbb R\cup\{+\infty\}$, which is closed and convex.

\begin{lemma}\label{lem:Bound}
 Let $\{g_k\}$ be a sequence of lsc, convex functions $g_k:\mathbb R^m\to\mathbb R\cup\{+\infty\}$ converging epi-graphically to $g:\mathbb R^m\to\mathbb R\cup\{+\infty\}$. Furthermore, let $\{z^k\}$ be an unbounded sequence such that $z^k\in\partial g_k(y^k)$ for all $k\in\mathbb N$ for some $\{y^k\}\to\bar y\in\dom g$. 
Then every accumulation point of $\Big\{\frac{z^k}{\|z^k\|}\Big\}$ lies in $N(\bar y\mid \dom g)$.
 \end{lemma}

\begin{proof} Let $\bar z$ be an accumulation point of $\Big\{\frac{z^k}{\|z^k\|}\Big\}$. W.l.g. we can assume that $\frac{z^k}{\|z^k\|}\to\bar z$. Moreover,  let $y\in\dom g$ be given. Since $\elim_{k\to\infty} g_k=g$, we may invoke \eqref{eq:EpiChar}
to obtain a sequence $\{\hat y^k\}\to y$ such that $\limsup_{k\to\infty} g_k(\hat y^k)\leq g(y)$.  Since, by assumption,  $z^k\in\partial g(y^k)$ for all $k\in\mathbb N$, we infer
$$
g_k(\hat y^k)-g_k(y^k)\geq (z^k)^T(\hat y^k-y^k)\quad\forall k\in\mathbb N.
$$ 
Dividing by $\|z^k\|$ yields 
$$
\frac{g_k(\hat y^k)-g_k(y^k)}{\|z^k\|}\geq \frac{(z^k)^T}{\|z^k\|}(\hat y^k-y^k)\to \bar z^T(y-\bar y).
$$
To prove the assertion it suffices to see that the numerator of the left-hand side of the above inequality is bounded from above at least on a subsequence. 
This, however, is true due to the choice of $\{\hat y^k\}$ and \eqref{eq:EpiChar}.
\end{proof}

\noindent 
A standard assumption in  the context of convex composite functions, cf. \cite{BuP 92}, is the  {\em basic contstraint qualification} which is formally stated in the following definition.

\begin{definition}[Basic constraint qualification]\label{def:BCQ} Let $f$ be given as in \eqref{eq:ConvCom}. Then $f$ is said to satisfy the {\em basic constraint qualification (BCQ)} at a point $\bar x\in \dom f$ if
$$
N( H(\bar x)\mid \dom g)\cap \nul H'(\bar x)^T=\{0\}.
$$
\end{definition}

\noindent
Note that, in the setting of \eqref{eq:ConvCom}, BCQ always holds at a point $\bar x\in\dom f$ where $H'(\bar x)^T$ has full column rank. Moreover, BCQ is always 
fulfilled when $g$ is finite-valued, since then $\dom g=\mathbb R^m$ and thus, 
$N(H(\bar x)\mid \dom g)=\{0\}$ for all $\bar x\in \mathbb R^n$. 
\\
The BCQ is important since it guarantees a rich subdifferential calculus for the composition $f=g\circ H$.

\begin{lemma}\cite[Theorem 10.6]{RoW 98}\label{lem:Chain}  
Let $f$ be given as in \eqref{eq:ConvCom}. If BCQ is satisfied at  $\bar x\in\dom f$,  
then $f$ is (subdifferentially) regular at   
$\bar x$ and we have 
$$
\p f(\bar x)=H'(\bar x)^T\partial g(H(\bar x)).
$$
\end{lemma}

\begin{theorem}\label{th:ConvComp} 
Let $s_g$ be an epi-smoothing function for $g$. 
If $s_f(\cdot,\cdot):=s_g(H(\cdot),\cdot)$ is an epi-smoothing function  for $f:=g\circ H$, then   
$$
\Limsup_{\mu\downarrow 0, x\to\bar x}\nabla_xs_f(x,\mu)=\p f(\bar x)
$$
for all $\bar x\in \dom f$ at which the BCQ holds.
\end{theorem}

\begin{proof} 
We need only show that
$\Limsup_{\mu\downarrow 0, x\to\bar x}\nabla_xs_f(H(x),\mu)\subset \p f(\bar x)$, 
since the $\Liminf$-inclusion is clear from Lemma \ref{lem:Gincl}.

To this end, let $v\in\Limsup_{\mu\downarrow 0, x\to\bar x}\nabla_xs_f(H(x),\mu)$ be given. Then there exist sequences $\{x^k\}\to\bar x$ and $\{\mu_k\}\downarrow 0$
 such that 
\begin{equation}\label{eq:Limit}
H'(x^k)^T\nabla_x s_g(H(x^k),\mu_k)=\nabla_x s_f (x^k,\mu_k)\to v.
 \end{equation}
Put $z^k:=s_g(H(x^k),\mu_k)\,(k\in\mathbb N)$.
If  $\{z^k\}$ were unbounded, then w.l.g. $\{\frac{z^k}{\|z^k\|}\}\to\bar z\neq 0$, and we infer from \eqref{eq:Limit} that
$$
\bar z\in \nul H'(\bar x)^T.
$$
On the other hand, Lemma \ref{lem:Bound} tells us that $\bar z\in N(H(\bar x)\mid \dom g)$, thus, 
$$
0\neq \bar z \in  N(H(\bar x)\mid \dom g)\cap \nul H'(\bar x)^T,
$$
which contradicts BCQ. Hence,  $\{z^k\}$ is bounded and converges at least on a subsequence, and due to Attouch's theorem \cite[Theorem 12.35]{RoW 98}
the limit (accumulation point) lies in $\partial g(H(\bar x))$. Using this  and the fact that $H'$ is continuous, we get
$$
v\in H'(\bar x)^T\partial g(H(\bar x))=\partial f(\bar x),
$$
where the equality is due to Lemma \ref{lem:Chain}. This concludes the proof.
\end{proof}

\begin{corollary}\label{cor:ConvComp}
Let $s_g$ be an epi-smoothing function for $g$, and suppose
$\omega$ is 1-coercive with $\omega(0)\leq 0$. Then
$s_f(\cdot,\cdot):=s_g(H(\cdot),\cdot)$ is an epi-smoothing function  for $f:=g\circ H$ and   
$$
\Limsup_{\mu\downarrow 0, x\to\bar x}\nabla_xs_f(x,\mu)=\p f(\bar x).
$$
for all $\bar x\in \dom f$ at which the BCQ holds.
\end{corollary}

\begin{proof}
The result follows immediately from Corollary \ref{cor:smooth} and Theorem \ref{th:ConvComp}.
\end{proof}

\noindent
We  point out that, unlike in the convex case in Theorem \ref{th:InfCon}, where we  obtain the gradient consistency condition directly
via Attouch's theorem, we cannot derive it in this case from a generalized version of Attouch's theorem for convex composite functions as it is presented in 
\cite[Theorem 2.1]{Pol 92}, since we do not meet the assumptions there.

\section{Constrained Optimization}\label{sec:ConOpt}

We now apply the results of the previous section the constrained optimization problem 
\begin{equation}\label{co}
\begin{array}{ll}
\mbox{minimize}&\phi(x)\\
\mbox{subject to}&h(x)\in C ,
\end{array}
\end{equation}
where $\map{\phi}{\bR^n}{\bR}$ and $\map{h}{\bR^n}{\bR^m}$ are smooth
mappings and $C\subset\bR^m$ is a nonempty closed convex set.
This is an example of a convex composite optimization problem 
\cite{Bur 85,Bur 87,BuP 92} where
the composite function $f=g\circ H$ is given by
\[
g(\gamma,y):=\gamma+\delta(y\mid C)\quad\mbox{and}\quad
H(x):=\begin{bmatrix}\phi(x)\\ h(x)\end{bmatrix} .
\]
In this case, $g$ is the sum of a smooth convex
function, $g_1(\gamma,y):=\gamma$, and a nonsmooth 
convex function $g_2(\gamma,y):=\delta(y\mid C)$. 
Hence, by Proposition \ref{prop:SmoothCalc}, we can obtain an epi-smoothing function for
$g$ by only smoothing the $g_2$ term. A straightforward computation shows that
\[
e_\mu g_2(y)=  \frac{1}{2\mu}\dist^2(y\mid C).
\]
Therefore, by Corollary \ref{cor:smooth},
\begin{equation}\label{cosf}
s_f(x,\mu)=\phi(x)+\frac{1}{2\mu}\dist^2(h(x)\mid C)
\end{equation}
is an epi-smoothing function for $f$. This is one of the classical smoothing
functions for constrained optimization \cite{FiM 87}. The BCQ becomes the condition 
\begin{equation}\label{cobcq}
\nul h'(x)^T\cap N(h(x)\mid C)=\{0\}.
\end{equation}
In the case where $C=\{0\}^s\times\bR^{m-s}_-$, the function \eqref{cosf} 
is the classical
least-squares smoothing function for nonlinear programming, and \eqref{cobcq}
reduces to the {\em Mangasarian-Fromovitz constraint qualification} (e.g., see \cite[Example 6.40]{RoW 98}).

Corollary \ref{cor:ConvComp} tells us that at every point $\bx$ with $h(\bx)\in C$ we have
\[
\Limsup_{\mu\downarrow 0, x\to\bar x}\nabla_xs_f(x,\mu)=\nabla \phi(\bx)+h'(\bx)^TN(h(\bx)\mid C) ,
\]
whenever condition \eqref{cobcq} holds at $\bx$, where, by Proposition \ref{moreau basics},
$$\nabla_xs_f(x,\mu)=\nabla \phi(x)+h'(x)^T\left(\frac{h(x)-\Pi_C(h(x))}{\mu}\right) .$$

\noindent
The results of  Section \ref{sec:ConvCom} 
allow us to make powerful statements about algorithms that
use the epi-smoothing function \eqref{cosf} to solve the optimization problem
\eqref{co}. We begin by studying the case of cluster points that are feasible for \eqref{co}.

\begin{theorem}\label{th:cosf}
Let $s_f$ be as in \eqref{cosf} with $\phi$, $h$, and $C$ satisfying the
hypotheses specified in \eqref{co}. Let $\{x^k\}\subset\bR^n$ and $\{\mu_k\}\downarrow 0$
satisfy $\norm{\nabla_xs_f(x^k,\mu_k)}\downarrow 0$.
Then every feasible cluster point $\bx$ of $\{x^k\}$  at which \eqref{cobcq} is satisfied, is a 
{\em Karush-Kuhn-Tucker point} for \eqref{co}, i.e., 
\[
0\in \partial f(\bx)=\nabla\phi(\bx)+h'(\bx)^T\ncone{h(\bx)}{C} .
\]
\end{theorem}

\begin{proof}
Lemma \ref{lem:Chain} implies that 
$\partial f(\bx)=\nabla\phi(\bx)+h'(\bx)^T\ncone{h(\bx)}{C}$.
Hence, by Corollary \ref{cor:ConvComp}, $\bx$ is a KKT point for \eqref{co}.
\end{proof}

\noindent
Theorem \ref{th:cosf} tells us that the feasible cluster points of sequences
of approximate stationary points of $s_f$ are KKT points, but, from and algorithmic
perspective, this does not give us a mechanism for testing proximity to
optimality via standard optimality conditions. That is, it does not show how to
approximate the multiplier vector. This is addressed by the following corollary.

\begin{corollary}
Let $s_f,\ \phi,\ h,\ C,\ \{x^k\},$ and $\{\mu_k\}$ be as in Theorem \ref{th:cosf},
and let $\bx$ be a cluster point of $\{x^k\}$ at which $h(\bx)\in C$ and \eqref{cobcq} is satisfied.
If $J\subset\mathbb{N}$ is a subsequence
for which $x^k\rightarrow_J\bx$, then the associated subsequence $\{y^k\}_J$,  where
$$
y^k:=\frac{h(x^k)-\Pi_C(h(x^k))}{\mu_k}\quad \forall k\in\mathbb N,
$$ 
remains bounded and every cluster point $\by$  is such that $(\bx,\by)$ is a 
{\em Karush-Kuhn-Tucker pair} for \eqref{co}, i.e.,
\[
0=\nabla \phi(\bx)+h'(\bx)^T\by\quad\mbox{with}\quad \by\in N(h(\bx)\mid C) .
\]
\end{corollary}

\begin{proof}
Let $J\subset\mathbb{N}$ and $\bx$ be as in the statement of the corollary.
Theorem \ref{th:cosf} tells us that $\bx$ is a KKT point for \eqref{co}, i.e.,  
$0\in \partial f(\bx)=\nabla\phi(\bx)+h'(\bx)^T\ncone{h(\bx)}{C}$.
 We first show that the subsequence $\{y^k\}_J$ given above is necessarily bounded. 

Suppose, to the contrary, that the sequence is not bounded. Then there is a further
subsequence $\hat{J}\subset J$ such that $\norm{y^k}\uparrow_{\hat{J}} +\infty$.
With no loss in generality we may assume that there is a unit vector $\ty$ such that
$y^k/\norm{y^k}\rightarrow_{\hat J}\ty$. Since $y^k\in\ncone{\Pi_C(h(x^k))}{C}$
for all $k$,
the outer semicontinuity of the normal cone operator $z\mapsto \ncone{z}{C}$ relative to $C$, 
cf. \cite[Proposition 6.6]{RoW 98}, implies that 
$\ty\in\ncone{h(\bx)}{C}$. Dividing $\norm{\nabla_xs_f(x^k,\mu_k)}$ by $\norm{y^k}$
and taking the limit over $\hat J$ gives $h'(\bx)^T\ty=0$. But this contradicts
the BCQ \eqref{cobcq} since $\ty$ is a unit vector. Therefore, the sequence $\{y^k\}_J$
is bounded. 

Let $\by$ be any cluster point of the sequence $\{y^k\}_J$
(at least one such cluster point must exist since this sequence is bounded).
As above, $\by\in\ncone{h(\bx)}{C}$, and by the hypotheses,
$0=\nabla \phi(\bx)+h'(\bx)^T\by$. Hence, $\bx$ is a KKT point for \eqref{co}
and $\by$ is an associated KKT multiplier.
\end{proof}

\noindent
We now address the case of infeasible cluster points, i.e., cluster points $\bx$ for
which $h(\bx)\notin C$. To understand this case, we must first review the subdifferential
properties of the distance function $\dist(\cdot\mid C)$ and the associated convex
composite function
\[
\psi(x):=\dist(h(x)\mid C) .
\]
First, recall from \cite[Proposition 3.1]{Bur 91} that
\begin{equation}\label{dist sd}
\partial \dist(y\mid C)=
\left\{\begin{array}{lcl}
\ncone{y}{C}\cap\bB &{\rm if}& y\in C,
\\
\ncone{y}{C+\dist(y\mid C)\bB}\cap\mathrm{bdry}(\bB) &{\rm if} & y\notin C,
\end{array}\right.
\end{equation}
where $\mathrm{bdry}(\bB)$ is the boundary of the unit ball, and, 
by \cite[Example 8.53]{RoW 98}, we also have
\begin{equation}\label{dist sd2}
\partial \dist(y\mid C)=\ncone{y}{C+\dist(y\mid C)\bB}\cap\mathrm{bdry}(\bB)
=\left\{\frac{y-\Pi_C(y)}{\dist(y\mid C)}\right\}\quad \forall\, y\notin C.
\end{equation}
In addition, from \cite[Equation 2.4]{Bur 87}, $\psi$ is
subdifferentially regular on $\bR^n$ with
\begin{equation}\label{psi sd}
\partial \psi(x)=h'(x)^T\partial \dist(h(x)\mid C) .
\end{equation}
These formulas yield the following result.

\begin{theorem}\label{th:cosf-inf}
Let $s_f,\ \phi,\ h,\ C,\ \{x^k\},$ and $\{\mu_k\}$ be as in Theorem \ref{th:cosf},
and let $\bx$ be a cluster point of $\{x^k\}$ at which $h(\bx)\notin C$.
Then $0\in\partial \psi(\bx)$.
\end{theorem}

\begin{proof}
Let $J\subset \bN$ be such that $x^k\rightarrow_J\bx$. Since 
$\norm{\nabla_xs_f(x^k,\mu_k)}\downarrow 0$, we have 
$\mu_k\norm{\nabla_xs_f(x^k,\mu_k)}\downarrow 0$, and consequently
\[
h'(x^k)^T(h(x^k)-\Pi_C(h(x^k)))\rightarrow 0.
\]
Hence, by the continuity of $\Pi_C$ and \eqref{dist sd2}, $0\in \partial \psi(\bx)$.
\end{proof}

\noindent
Theorem \ref{th:cosf-inf} shows that any algorithm that drives $\nabla_xs_f(x^k,\mu_k)$
to zero as $\mu_k\downarrow 0$ performs admirably even when the problem \eqref{co}
is itself infeasible. That is, in the absence of feasibility, it naturally tries to locate a
{\em nonfeasible stationary point} for \eqref{co} as defined in \cite{Bur 87b}. 
It may happen that the original problem is feasible while all cluster points are
nonfeasible stationary points. This can be rectified by placing a further restriction
on how the iterates $\{x^k\}$ are generated.

\begin{proposition}\label{prop:feasible}
Let $C,\ \phi,\ h$, and $s_f$ be as in \eqref{co} and \eqref{cosf}, and let
$\mu_k\downarrow 0$.
Suppose that there is a known feasible point $\tx$ for \eqref{co}.  
If $\{x^k\}$ is a sequence for which 
$s_f(x^k,\mu_k)\le s_f(\tx,\mu_k)=\phi(\tx)$ for all $k=1,2,\dots$,
then every cluster point of $\{x^k\}$ must be feasible for \eqref{co}.
\end{proposition}

\begin{proof}
Let $\bx$ be a cluster point of $\{x^k\}$ and let $J\subset\bN$ be such that $x^k\rightarrow_J\bx$.
If $\bx$ is not feasible, then $\frac{1}{2\mu_k}\dist^2(h(x^k)\mid C)\rightarrow_J+\infty$. But
$s_f(x^k,\mu_k)=\phi(x^k)+\frac{1}{2\mu_k}\dist^2(h(x^k)\mid C)\le \phi(\tx)$ giving
the contradiction $\phi(x^k)\rightarrow_J-\infty$.
\end{proof}

\noindent
In fact, without further hypotheses, feasibility might not be attained in the limit even in the prototypical example of convex composite optimization, 
the Gauss-Newton method for solving nonlinear systems of equations. 
It is often the case that the additional hypotheses employed 
are related to the BCQ \eqref{cobcq}.
One way to understand the role of nonfeasible stationary points and their effect on computation
is through constraint qualifications that apply to nonfeasible points. These constraint
qualifications extend \eqref{cobcq}
to points on the whole space.
Among the many possible extensions one might consider,
we use one from the geometry of the subdifferential in \eqref{dist sd}.
We say that the {\em extended constraint qualification} (ECQ) for \eqref{co} 
is satisfied if
\begin{equation}\label{ecq}
\nul h'(x)^T\cap\ncone{h(x)}{C+\dist(h(x)\mid C)\bB}=\{0\} .
\end{equation}
Note that this condition is well defined on all of $\bR^n$ and reduces to \eqref{cobcq} when
$h(x)\in C$. When $h(x)\notin C$, it is easily seen that $0\in \partial \psi(x)$ if and only if
\eqref{ecq} is not satisfied. Hence, if one assumes that ECQ is satisfied at all iterates, then
nonfeasible cluster points cannot exist. For example, if $C=\{0\}$, then a standard global
constraint qualification is to assume that $h'(x)$ is everywhere surjective, i.e., $\nul h'(x)^T=\{0\}$
for all $x$. This implies \eqref{ecq} which simply says that $h'(x)^Th(x)\ne 0$ whenever
$h(x)\ne 0$ and $h'(x)$ is surjective whenever $h(x)=0$.

\section{Final Remarks}\label{sec:FinRem}

In this paper we have synthesized the infimal convolution smoothing ideas proposed by
Beck and Teboulle in \cite{BeT 12} with the notion of gradient consistency
defined by Chen in \cite{Che 12}. To achieve this we make use of epi-convergence techniques
that are well suited to the study of the variational properties of parametrized families of functions.
Using epi-convergence, we defined the notion of epi-smoothing  for which we
established a rudimentary calculus. Epi-smoothing is a weakening of the kinds of smoothing
studied in \cite{BeT 12} where the focus is on convex optimization and the derivation of
complexity results which necessitate stronger forms of smoothing.
We then applied the epi-smoothing ideas to study the
epi-smoothing properties of convex composite functions, a very broad and important class
of nonconvex functions. In particular, we showed that general constrained optimization falls
within this class. Using the epi-smoothing calculus, we easily derived the convergence properties
of a classical smoothing approach to constrained optimization establishing the convergence 
properties even in the case when the underlying optimization problem is not feasible.
This application demonstrates the power of these ideas as well as their ease of use. 

\section{Acknowledgements} The authors thank Prof. Christian Kanzow for helpful comments on earlier
drafts of the paper.

\end{document}